\documentclass{amsart}

\usepackage{a4wide,latexsym,amssymb,amsthm,amsmath,color}

\theoremstyle{plain}

\newtheorem{theorem}{Theorem}

\newtheorem{proposition}{Proposition}

\theoremstyle{remark}

\makeatletter
\def\Ddots{\mathinner{\mkern1mu\raise\p@
		\vbox{\kern7\p@\hbox{.}}\mkern2mu
		\raise4\p@\hbox{.}\mkern2mu\raise7\p@\hbox{.}\mkern1mu}}
\makeatother

\title[ Warped product hypersurfaces ]
{Warped product hypersurfaces in the pseudo-Euclidean space }

\author[Moruz]{Marilena Moruz}
\address{Faculty of Mathematics,  Alexandru Ioan Cuza University of Ia\c si, Bulevardul Carol I, Nr.11, 700506, Ia\c si, Rom\^ania}
\email{marilena.moruz@gmail.com}

\begin{document}


\maketitle

\begin{abstract}
We study hypersurfaces in the pseudo-Euclidean space $\mathbb{E}^{n+1}_s$, which write as a warped product of a $1$-dimensional base with an $(n-1)$-manifold of constant sectional curvature. We show that either they have constant sectional curvature or they are contained in a rotational hypersurface. Therefore,  we first define rotational hypersurfaces in the pseudo-Euclidean space.
\end{abstract}

\section{Introduction}
Warped product manifolds are a generalization of product manifolds. The metric on such a manifold is a warped product metric, which roughly speaking, rescales the metric on the second factor of the product manifold.\\
Warped product manifolds play an important role in physics as well (\cite{5}, \cite{beem}), not only in mathematics. For instance,  the space around a massive star or black hole is given as a warped product in the Schwarzschild space-time relativistic model (\cite{2}). Also in string theory, in the context of the classical instability of extra spatial dimensions, warped product manifolds provide a more accurate description of the spacetime than just product manifolds (\cite{penrose}). Moreover, a reasonably realistic cosmological model is the Robertson-Walker spacetime model. It is given as a warped product manifold $(I\times_fM^3(\kappa), g)$, $g=-dt^2+f^2(t)g_{M^3}$, where $I$ is a real interval and $M^3$ is a $3$-dimensional Riemannian manifold with constant sectional curvature $\kappa$ and Riemannian metric $g_{M^3}$. The warping function $f$ describes the expanding or contracting of the Universe. \\
An interesting point of view for the study of warped product manifolds has been  their immersability in different ambient spaces, such as Riemannian manifolds or K\"ahler manifolds.  M. do Carmo and M. Dajczer have extended in 1983 the rotation surfaces of the three-dimensional Euclidean space $\mathbb{E}^3$ to rotation hypersurfaces in spaces of constant sectional curvature of arbitrary dimension. In this sense, roughly speaking, a rotation hypersurface $M^n$ in a real space form $\tilde M^{n+1}(\kappa)$ is generated by moving an $(n-1)$-dimensional submanifold $\Sigma\subset \tilde M^{n+1}(\kappa)$ along some curve. The authors of \cite{docarmo} find sufficient conditions for hypersurfaces in real space forms to be rotational hypersurfaces and give examples of such rotational hypersurfaces. One of these conditions is expressed in terms of the principal curvatures of the immersed hypersurface and states that if the principal curvatures $k_1,k_2,\ldots,k_n$ satisfy $k_1=\ldots=k_{n-1}=-\lambda$, $k_n=-\mu=-\mu(\lambda)$ and $\lambda\neq\mu$, then the immersed hypersurface is contained in a rotation hypersurface, provided $n\geq3$ (see Theorem 4.2. in \cite{docarmo}).\\
\indent In the present paper we study hypersurfaces $M^n$ in the pseudo-Euclidean space $\mathbb{E}^{n+1}_s$, which write as a warped product hypersurface of the form $M^n=I\times_f \bar M^{n-1}(c)$, where  $f$ is a non-constant function, $I\subset \mathbb R$ is an open interval and $\bar M^{n-1}(c)$ has constant sectional curvature $c$. The same problem has been investigated in \cite{MV} for the case when the ambient space is a pseudo-Riemannian space form of non-zero sectional curvature.  M. Dajczer and R. Tojeiro  give a characterization  in \cite{Tojeiro2} for such hypersurfaces in the positive definite case. We address this problem in the general situation when the metric in the ambient space is indefinite, of signature $s\geq 0$.  We define rotational hypersurfaces in the pseudo-Euclidean space and show that warped product hypersurfaces of the form mentioned above  either have constant sectional curvature or are rotational hypersurfaces. We prove the following theorem.
 
\begin{theorem}\label{The}
	Let  $M^n=I\times_f \bar M^{n-1}(c)$ be a warped product of a real interval $I\subset \mathbb{R}$ and an $(n-1)$-dimensional real space form $\bar M^{n-1}(c)$ with constant sectional curvature $c$, for a positive, non-constant function $f$ on $I$. Suppose that 
	$$F:M^n\longrightarrow \mathbb{E}^{n+1}_s$$ 
	is a non-degenerate isometric immersion into the $(n+1)$-dimensional pseudo-Euclidean space, endowed with an indefinite metric of signature $ s\geq 0$.\\
	Then $M^n$ has constant sectional curvature or is a rotational hypersurface. 
\end{theorem}
We organize this paper as follows. In section 2 we give preliminary notions on warped product manifolds and hypersurfaces in the pseudo-Euclidean space. In particular, we define what rotational hypersurfaces are in this setting. In the next section, we prove Theorem \ref{The}.

\section{Preliminaries} \label{prelim}
\textbf{Hypersurfaces in the pseudo-Euclidean space $\mathbb{E}^{n+1}_s$.} Let $M^n$ be a hypersurface isometrically immersed in the pseudo-Euclidean space $\mathbb{E}^{n+1}_s$. Denote by $X,Y$ tangent vector fields and by $\xi$ the normal vector field on $M^n$. Then the Gauss and Weingarten formulas write out as
\begin{align*}
	 D_XY&=\nabla_XY + h(X,Y), \\
	 D_X\xi &=-SX, 
\end{align*} 
where $D$ and $\nabla$ are the Levi-Civita connections on $\mathbb{E}^{n+1}_s$ and $M^n$, respectively,  $h$ is the second fundamental form of $M^n$ and $S:=A_{\xi}$ is the shape operator of $\xi$. It holds that $\langle h(X,Y),\xi \rangle=\langle S X,Y \rangle$. Let $R$ be the Riemannian curvature tensor of $M^n$. The equations of Gauss and Codazzi are 
\begin{align}
\langle R(X,Y)Z,W\rangle &=\langle h(Y,Z), h(X,W) \rangle -\langle h(X,Z),h(Y,W) \rangle ,\label{gauss}\\
(\nabla h)(X,Y,Z)&=(\nabla h)(Y,X,Z)\label{cdz2},
\end{align} 
where the covariant derivative of $h$ is given by
\begin{equation*}
(\nabla h)(X,Y,Z)=\nabla^{\perp}_X h(Y,Z)-h(\nabla_XY,Z)-h(Y,\nabla_XZ),
\end{equation*}for $\nabla^{\perp}$ the normal connection of $M^n$ in  $\mathbb{E}^{n+1}_s$.
The following Ricci identity holds
\begin{align}\label{ricci}
\begin{array}{l}
(\nabla^2 h)(X,Y,Z,W)-(\nabla^2 h)(Y,X,Z,W)=
h(R(X,Y)Z,W)-h(Z,R(X,Y)W),
\end{array}
\end{align}
where the second covariant derivative of $h$ is defined as 
\begin{align}\label{defsecder}
(\nabla^2 h)( X,Y,Z,W)=&\nabla^{\perp}_X(( \nabla h)(Y,Z,W))-(\nabla h)(\nabla _XY,Z,W)\\
& -(\nabla h)(Y,\nabla _XZ,W)-(\nabla h)(Y,Z,\nabla _XW ).\nonumber
\end{align}

\textbf{Rotational hypersurfaces in the pseudo-Euclidean space  $\mathbb{E}^{n+1}_s$.}
Let $\mathbb{E}^{n+1}_{s}$ denote the Euclidean space $\mathbb{R}^{n+1}$ endowed with the pseudo-Euclidean metric $\langle \cdot, \cdot\rangle$ of signature $s$.
In the following part we define rotational hypersurfaces in   $\mathbb{E}^{n+1}_{s}$ with a discussion on the type of axis of rotation. We recall that a tangent vector $v$ on $M^n$ is called \emph{spacelike} if $\langle v,v \rangle >0$ or $v=0$, \emph{timelike} if $v<0$  and \emph{null} (or $lightlike$) if $\langle v,v\rangle=0$ and $v\neq 0$.\\
Let $\alpha$ be a planar curve included in a non-degenerated $2$-plane $\Pi^2\subset\mathbb{E}^{n+1}_s$. Let $d\in \Pi^2$ be a straight line. A rotational hypersurface in $\mathbb{E}^{n+1}_{s}$ is given by the image of $\alpha$ under the action of isometries in $\mathbb E^{n+1}_s$ which keep $d$ fixed. \\
Let $\{e_1, e_2,\ldots, e_{n+1}\}$ be an orthogonal basis of $\mathbb E^{n+1}_s$, for which the lengths of the vectors are $\pm 1$. We may assume,  if necessary after reordering the vectors, that the $s$ $``-"$ signs of the metric correspond to the last $s$ vectors in the basis. So, a way to express the metric is $(+,+,\ldots,-,-)$. \\

\textbf{Case 1.} Assume the axis of rotation is spacelike. Without loss of generality, we may suppose that the axis is in the direction of $e_1$. However, the metric restricted to $\Pi^2$ may be positive definite or indefinite. Therefore, we may suppose $\Pi^2$ is spanned by $e_1$, $e_2$ or by $e_1$, $e_{n+1}$ and so we have that a curve contained in $\Pi^2$ has one of the forms:
\begin{equation}
\alpha(t)=(f_1(t),f_2(t),0,\ldots,0)\quad\text{or}\quad \alpha(t)=(f_1(t),0,\ldots,0,f_2(t)),
\end{equation} for $f_1$, $f_2$ real functions of one variable.
An isometry of $\mathbb{E}^{n+1}_s$ fixing the axis is given by 
\begin{equation}
\begin{split}
\left(
\begin{array}{ c c c c }
1 & 0   &\ldots &0    \\
0 & v_2 &\ldots &u_2  \\
\vdots  &\vdots&  & \vdots \\
0 & v_{n+1} &\ldots &u_{n+1} 
\end{array} \right),
\end{split}
\end{equation}
where the columns are pseudo-orthogonal vectors, i.e. $\sum\limits_{i=2}^{n+1} \varepsilon_i v_i^2=1$,  $\sum\limits_{i=2}^{n+1} \varepsilon_i u_i^2=-1$, where  $\varepsilon_i=\pm 1$ and are determined according to the signature of the metric in each case. Therefore, a rotational hypersurface is described by one of the parameterizations:  
\begin{equation}\label{primulcaz}
\alpha(t)=(f_1(t),f_2(t)v_2,\ldots,f_2(t)v_{n+1})\quad\text{or}\quad \alpha(t)=(f_1(t),f_2(t)u_2,\ldots,f_2(t)u_{n+1}).
\end{equation}

\textbf{Case 2.} Assume the axis of rotation is timelike. We may suppose without loss of generality that the axis is in the direction of $e_{n+1}$.
Since the metric restricted to $\Pi^2$ may be negative definite or indefinite,  we may assume that $\Pi^2$ is spanned by $e_1$, $e_{n+1}$ or by $e_n$, $e_{n+1}$ and so, we have that a curve contained in $\Pi^2$ has one of the forms:
\begin{equation}
\alpha(t)=(f_1(t),0,\ldots,0,f_2(t))\quad\text{or}\quad \alpha(t)=(0,\ldots,0,f_1(t),f_2(t)),
\end{equation} for $f_1$, $f_2$ real functions. \
An isometry of $\mathbb{E}^{n+1}_s$ fixing the axis is given by 
$ 
\left(
\begin{array}{cccc}
v_1 &\ldots &u_1  & 0\\
\vdots & &\vdots & \vdots \\
v_{n} &\ldots &u_{n}& 0 \\
0   & \ldots & 0    & 1
\end{array} \right),
$
where the columns are pseudo-orthogonal vectors, i.e. $\sum\limits_{i=1}^{n} \varepsilon_i v_i^2=1$,  $\sum\limits_{i=1}^{n} \varepsilon_i u_i^2=-1$, for  $\varepsilon_i=\pm 1$ according to the signature of the metric. Therefore, a rotational hypersurface is described in this case by one of the parameterizations:  
\begin{equation}
\alpha(t)=(f_1(t)v_1,\ldots,f_1(t)v_{n},f_2(t))\quad\text{or}\quad \alpha(t)=(f_1(t)u_1,\ldots,f_1(t)u_{n},f_2(t)).
\end{equation}

\textbf{Case 3.} Assume the axis of rotation is null. Then we must have that the metric restricted to $\Pi^2$ is of the form 
$
\left(
\begin{array}{ c c }
+ & \star \\
\star & -
\end{array} \right)
$ 
and we may see $\Pi^2$ to be generated by $e_1$ and $e_{n+1}$. Let $n_1=(1,0,\ldots,0,1)$ and $n_2=(1,0,\ldots,0,-1)$ be two null vectors spanning $\Pi^2$, such that we may assume without loss of generality that $n_1$ is in the direction of the axis and stays fixed under isometries. Hence,  the profile curve of the rotational hypersurface is described as $t\mapsto f_1(t)n_1+f_2(t)n_2$, where $f_1,\, f_2$ are real functions of one variable.\\
Let $v:=\mathcal O n_2$, where $v=(v_1,\ldots,v_{n+1})\in \mathbb R^{n+1}_s$ and $\mathcal O$ is some isometry of $\mathbb R^{n+1}_s$. Since $\langle n_2,n_2\rangle =0$, $\langle n_1,n_2 \rangle=2$ and the isometries preserve the metric, we have that $v_1=1-\frac{1}{4}\sum\limits_{i=2}^n\varepsilon_i v_i^2$, $v_{n+1}=-1-\frac{1}{4}\sum\limits_{i=2}^n\varepsilon_i v_i^2$.
Therefore, a rotational hypersurface in this case is described by 
\begin{equation*}
(t,v_2,\ldots,v_n)\mapsto
f_1(t) \left(\begin{matrix}
1\\0 \\ \vdots\\ 0\\1\\
\end{matrix} \right) +
f_2(t) \left(\begin{matrix}
1-\frac{1}{4}\sum\limits_{i=2}^{n}\varepsilon_i v_i^2 \\ v_2\\ \vdots\\v_{n}\\-1-\frac{1}{4}\sum\limits_{i=2}^{n}\varepsilon_i v_i^2
\end{matrix} \right),
\end{equation*}
for $f_1,\, f_2$  real functions of one variable.

\textbf{Warped product manifolds (\cite{chen}).} Let $B$ and $F$ be two pseudo-Riemannian manifolds of positive dimensions, endowed with pseudo-Riemannian  metrics $g_B$ and $g_F$, respectively. Let $f$ be a positive smooth function on $B$. On the product manifold $B\times F$, on which we have the natural projections 
$$ \pi: B\times F\longrightarrow B \text{ and } \eta: B\times F\longrightarrow F,$$
we consider the pseudo-Riemannian structure such that 
\begin{equation}\label{metric}
\langle X,Y\rangle=\langle \pi_{\star}(X),\pi_{\star}(Y)\rangle + f^2({\tiny\pi(x)}) g( \eta_{\star}(X), \eta_{\star}(Y) ), 
\end{equation} for any tangent vectors $X,Y\in TM$. 
This way $B\times F$ becomes a warped product manifold $M=B\times_f F$, for which the metric $g$ satisfies $g=g_B+f^2g_F$.\\
The function $f$ is called the warping function of the warped product, $B$ is called the base and $F$, the fiber. If $f$ is constant, then the warped product $B\times_f F$ is called trivial. In this case $B\times_f F$ is the Riemannian product $B\times F_f$, where $F_f$ is the Riemannian manifold equipped with the metric $f^2 g_F$.\\
The leaves $B\times \{q\}=\eta^{-1}(q)$ and the fibers $\{p\}\times F=\pi^{-1}(p)$ are Riemannian submanifolds of $M$. Vectors tangent to leaves are called horizontal and those tangent to fibers are called vertical. If $u\in T_pB$, $p\in B$ and $q\in F$, then the lift $\bar{u}$ of $u$ to $(p,q)$ is the unique horizontal vector in  $T_{(p,q)}M$ such that $\pi_{\star}(\bar{u})=u$. For a vector field $X\in \mathfrak{X}(B)$, the lift of $X$ to $M$ is the vector field $\bar{X}$ whose value at $(p,q)$ is the lift of $X_p$ to $(p,q)$.  The set of all horizontal lifts is denoted by $\mathcal{L}(B)$. Similarly, we denote by $\mathcal{L}(F)$ the set of all vertical lifts.\\
One may express the curvature tensor of the warped product manifold in terms of the warping function and the curvature tensors of the two components $B$ and $F$.  For a warped product manifold $M=B\times_fF$, we define the lift $\tilde T$ of a covariant tensor $T$ on $B$ to $M$ as the pullback $\pi^*(T)$ via the projection $\pi:M\longrightarrow B$. We will denote by $^{B}\!R$ and $^{F}\!R$  the curvature tensors of $B$ and $F$, respectively. Let  $\nabla$ denote the Levi-Civita connection of $B\times_fF$. We have the following propositions from \cite{chen}.
\begin{proposition}\label{prop}
	Let  $M=B\times_fF$ be a warped product of two pseudo-Riemannian manifolds. If $X,Y,Z\in \mathcal{L}(B)$  and $U,V,W\in \mathcal{L}(F)$, then we have
	\begin{itemize}
		\item[(1)] $R(X,Y)Z\in \mathcal{L}(B)$ is the lift of ${}^{B}\!R(X,Y)Z$ on $B$;
		\item[(2)] $R(X,V)Y=\frac{H^f(X,Y)}{f}V$;
		\item[(3)] $R(X,Y)V= R(V,W)X=0$;
		\item[(4)] $R(X,V)W=-\frac{\langle V,W\rangle}{f}\nabla_X(\nabla f)$;
		\item[(5)] $R(V,W)U= {}^F\!R(V,W)U+\frac{\langle \nabla f, \nabla f \rangle}{f^2}\{\langle V,U \rangle W-\langle W,U \rangle V\}$,
	\end{itemize}
where $R$ is the curvature tensor of $M$ and $H^f$ is the Hessian of $f$. 
\end{proposition}
\begin{proposition}\label{propnabla}
	For $X,Y\in \mathcal{L}(B)$ and $V,W\in \mathcal{L}(F)$, we have on $B\times_f F$ that
	\begin{enumerate}
		\item $\nabla_X  \mathcal{L}(B)$ is the lift of $\nabla_XY$ on $B$;
		\item $\nabla_XV=\nabla_VX=(X \ln f)V$;
		\item $nor(\nabla_VW)= -\frac{\langle V,W\rangle}{f}\nabla f$;
		\item $tan (\nabla_VW)\in\mathcal{L}(F)$ is the lift of $\nabla'_VW$ on $F$, where $\nabla '$ is the Levi-Civita connection on $F$.
 	\end{enumerate}
\end{proposition}

\section{Proof of Theorem \ref{The}}
\noindent Let $M^n$ be a warped product hypersurface in the pseudo-Euclidean space, of the form 
$$ M^n=I \times_f \bar M^{n-1} (c)	\overset{F}{\longmapsto} \mathbb{E}^{n+1}_s, $$
where  $I$ is a real interval, $f$ is a function defined on $I$ and the second factor $\bar M^{n-1}(c)$ is a real space form with constant sectional curvature $c$, endowed with the pseudo-Riemannian metric $g$ and of dimension $n-1$. By $\langle\cdot,\cdot\rangle$ we denote the metric of signature $s$ in the pseudo-Euclidean space.
Let $E_1=\frac{\partial}{\partial t}$ be the unit vector tangent to $I$, such that $\langle E_1, E_1\rangle= 1.$ Notice that we can always assume $E_1$ to have positive length, since, if necessary, we can change the signature of the metric on $\mathbb E^{n+1}_s$ from $s$ to $n+1-s$.
Denote by $\xi$ the normal to the hypersurface and let $\tilde \varepsilon $ such that:
$$ \langle\xi,\xi \rangle=\tilde{\varepsilon}=\pm 1. $$
Let $ U, V, W$ be tangent vectors to $\bar M$. Then from equation \eqref{metric} it follows that:
$$
\langle  E_1, U\rangle=0,\quad
\langle  U, V\rangle=f^2 g( U, V)
$$
and from Proposition \ref{prop} we know that the curvature tensor on $M$ is given as follows:
\begin{align}\label{curvatureM}
\begin{split}
&R(E_1,V)E_1=\frac{f''}{f}V,\\
&R(W,V)E_1=0,\\
&R(E_1,V)W=-\langle V,W \rangle \frac{f''}{f}E_1,\\
&R(V,W)U={}^{\bar M}\!R(V,W)U+\frac{f'^2}{f^2}\{\langle V,U \rangle W-\langle W,U\rangle V\},
\end{split}
\end{align}
where ${}^{\bar M}\!R$ is the curvature tensor on $\bar M$ and is given by 
\begin{equation}\label{Gauss}
{}^{\bar M}\!R(U,V)W=c(g(V,W)U-g(U,W)V).
\end{equation}
In order to investigate our hypersurfaces we will use a method initially introduced in \cite{tsinghuapaper},  which the authors called \emph{the Tsinghua principle}. The following computations describe this method, which, eventually brings the advantage of providing a new relation for the data on the hypersurface, see equation \eqref{eq}. Let $X,Y,Z,W$ be tangent vector fields on $M^n$. First, we take the derivative with respect to $W$ in the Codazzi equation \eqref{cdz2}. This gives $$(\nabla^2h)(W,X,Y,Z)-(\nabla^2h)(W,Y,X,Z)=0.$$ We then cyclically permute $W,X,Y$ in this relation and use additionally the Ricci identity as in \eqref{ricci}, together with the first Bianchi identity to obtain
\begin{align}\label{eq}
0=h(Y,R(W,X)Z)+h(W,R(X,Y)Z)+h(X,R(Y,W)Z).
\end{align}
Let $W,X,Z\in T_q\bar M$, for $q\in \bar M$ and let $Y=E_1$. Then equation \eqref{eq} implies that 
\begin{equation}\label{new}
0=\left( c +f''f-f'^2\right) \Big(\langle X,Z \rangle h(W,E_1)-\langle W,Z \rangle h(X,E_1)\Big),
\end{equation}
which yields two cases, according to whether $c+f''f-f'^2$ vanishes or not.\\

\textbf{Case 1.} Assume 
\begin{equation}\label{case1}
c+f''f-f'^2=0.
\end{equation}
We shall see that $M^n$ has constant sectional curvature.\\
For more convenience, in this case we prefer to denote by $E_1$ the unit vector spanning the tangent space to $I$ and by $X_i$, $i=2,\ldots,n$ the orthonormal vectors spanning the tangent space to $\bar M$, with $\langle X_i, X_i\rangle =\pm 1$. Let $\Pi_{1i}=span\{E_1,X_i\}$ and $\Pi_{ij}=span\{X_i,X_j\}$ be $2$-dimensional planes in the tangent space of $T_xM^n$, for $x\in M^n$. We get immediately that the sectional curvatures of  $M^n$ for all such planes are given by $K(\Pi_{1i})=-\frac{f''}{f}$ and  $K(\Pi_{ij})=\frac{c}{f^2}-\frac{f'^2}{f^2}$.\\
By \eqref{case1}  we have $-\frac{f''}{f}=\frac{c}{f^2}-\frac{f'^2}{f^2}$ and we can prove straightforwardly that $\frac{f''}{f}$ is constant.  Hence, $M^n$ has constant sectional curvature. \\

\textbf{Case 2.} Assume $$c+f''f-f'^2\neq 0.$$
From \eqref{new} we have that  $$\langle X,Z \rangle h(W,E_1)-\langle W,Z \rangle h(X,E_1)=0. $$ 
Next, we choose $Z$ such that $\langle X,Z \rangle=0$ and $\langle W,Z\rangle\neq 0$. It follows that $h(X,E_1)=0$, for all $X\in T_q\bar M$, i.e. $SE_1\perp X$. Then there exists $\mu$ function of $t$ such that
\begin{align}
SE_1=\mu E_1
\end{align}
and we have $h(E_1,E_1)=\tilde{\varepsilon}\mu\xi$. We use the curvature relations in \eqref{curvatureM} and the Gauss equation in \eqref{gauss} to express  $R(E_1,X)E_1$. We obtain 
\begin{equation}\label{unu1}
SX=-\frac{1}{ \tilde{\varepsilon} \mu } \frac{f''}{f} X.
\end{equation}
We do similarly for $R(X,Y)X$, with $X,Y$ tangent vectors to $\bar M$, and obtain
\begin{equation}\label{unu}
(\langle X,X \rangle Y -\langle X,Y\rangle X)\left(\frac{f'^2}{f^2}-\frac{c}{f^2} \right)= \langle SY,X\rangle SX-\langle SX,X\rangle SY.
\end{equation}
In order to understand equations \eqref{unu1} and \eqref{unu}, we need to determine the form of the shape operator of $\bar M$. We prove the following proposition.
\begin{proposition}
	There exists a basis $\{X_i\}_{i=1,\ldots,n-1}$ in $T_q\bar M$ for which the shape operator restricted to $\bar M$ is given by $SX_i=\lambda X_i$, for real functions $\lambda=\lambda(t)$.
\end{proposition}
\begin{proof}
 We choose first a convenient notation for an orthonormal basis on $T_{q}\bar M$, with vectors of length $\pm 1$: $\{X_1,\ldots,X_{n-1}\}$.  We will use the result in \cite{petrov}, which states that for the shape operator $S$, which is symmetric with respect to the metric $\langle \cdot, \cdot\rangle$ (and so with respect to $g(\cdot,\cdot)$ too) there exists a basis on the tangent space of $M^n$ such that we have, for some natural positive numbers $m_i\leq m\leq n$, the following:
\begin{equation}
S=\left(
\begin{array}{ c c c}
S_1 & & \\ & \ddots & \\ & & S_m 
\end{array}
\right),
\quad 
g= \left(
\begin{array}{ c c c}
g_1 & &\\
& \ddots &\\
& & g_m
\end{array}
\right),
\end{equation}
where 
\begin{equation}\label{blocuri}
S_i=\left(
\begin{array}{ c c c c}
\lambda_i & 1 & & \\
&\lambda_i & \ddots & \\
 & & \ddots& 1\\
 &  & &\lambda_i
\end{array}
\right)_{m_i\times m_i},
\quad 
g_i= \left(
\begin{array}{  c c c}
  & & \epsilon_i\\
  &  \Ddots &\\
\epsilon_i & & 
\end{array}
\right)_{m_i\times m_i},\, \mid \epsilon_i\mid=1.
\end{equation}
Notice that the eigenvalues $\lambda_i$ may be complex, in which case $\epsilon_i\in \mathbb{C}$.\\
We will prove by contradiction that the shape operator $S$ cannot contain blocks $S_i$ of dimension $m_i>1$.
Assume first that there exists such a block of dimension $3$. Without loss of generality we may assume that $m_1=3$. Then, according to \eqref{blocuri}, we may write
\begin{align*}
SX_1&=\lambda_1 X_1,\\
SX_2&=X_1+\lambda_1 X_2,\\
SX_3&=X_2+\lambda_1 X_3.
\end{align*}
In this case, using the symmetry of the shape operator, we find that $\langle X_2,X_2\rangle =\langle X_1,X_3\rangle$, whereas we know from the expression of the metric that $ \langle X_1,X_3\rangle\neq 0$. Moreover, the equation following from \eqref{unu} for $R(X_2,X_3)X_2$ gives that $\langle X_2,X_2\rangle =0$, which is a contradiction. Hence there cannot be blocks of dimension $3$.\\
When there is at least one block of dimension $m_i$ higher than $3$, we prove similarly a contradiction by picking $X_{\alpha}, X_{\beta}$, $1\leq \alpha, \beta<n$ in \eqref{unu} to be given either by $\alpha=\frac{m_i}{2}, \beta=\alpha+2$ if $m_i$ is even, or by $\alpha=[\frac{m_i}{2}]+1, \beta=\alpha+1$ if $m_i$ is odd. Therefore $m_i\in\{1,2\}$.\\
We shall now prove there cannot be any blocks of dimension $2$ neither. Assume there exists at least one block of dimension $2$ in the matrix of the shape operator. We may assume, without loss of generality, that $m_1=2$.  Therefore
\begin{align*}
SX_1&=\lambda_1 X_1,\\
SX_2&=X_1+\lambda_1 X_2.
\end{align*}
We can write equation \eqref{unu} for $R(X_1,X_2)X_1$ and  it follows that $\lambda_1^2=\frac{c}{f^2}-\frac{f'^2}{f^2}$.\\
Similarly, for $R(X_2,X_1)X_2 $ we obtain that $\lambda_1=0$. Since we see easily that $\frac{c}{f^2}-\frac{f'^2}{f^2}$  has non-zero derivative, and hence it is not zero, we obtain a contradiction.\\
Therefore, there are only blocks of dimension $1$ in the matrix of the shape operator and so $$SX_i=\lambda_i X_i, \, i=1,\ldots,n-1.$$
Furthermore, we use \eqref{curvatureM} and  \eqref{Gauss} to evaluate $R(E_1,X_i)E_1$ and $R(X_j,X_i)X_j$. It follows,  respectively, that 
\begin{align}\label{generaleq}
\begin{split}
 - \tilde{\varepsilon}\mu\lambda_i&=\frac{f''}{f},\\  
 \tilde{\varepsilon} \lambda_i\lambda_j &=-\frac{f'^2}{f^2}+\frac{c}{f^2}.
\end{split}
\end{align}
It is straightforward to check from the second relation in \eqref{generaleq} that for indexes $1\leq i,j < n$ we must have  $\lambda_i=\lambda_j$ and we also notice that the eigenvalues of $X_i$ are real. 
\end{proof}

After determining the form of the shape operator,  the expressions of the curvature tensor of the warped product submanifold allowed us to obtain equations \eqref{generaleq}, which, if we denote $\lambda:=\lambda_i$, become
\begin{align}\label{relscurta}
\begin{split}
 - \tilde{\varepsilon}\mu\lambda&=\frac{f''}{f},\\  
 \tilde{\varepsilon} \lambda^2 &=-\frac{f'^2}{f^2}+\frac{c}{f^2}.
\end{split}
\end{align}
Moreover, using Proposition \ref{propnabla} and the form of the shape operator, we compute
\begin{align}
\begin{split}
D_{E_1}\xi=-\mu E_1,\quad &D_X\xi=-\lambda X, \quad D_XE_1=\frac{f'}{f}X,\\
D_{E_1}E_1&=h(E_1,E_1)=\tilde{\varepsilon}\mu\xi.\\
\end{split}
\end{align}
In order to determine the immersion of the submanifold, there is one last step remaining, for which we will consider two cases, according to the length of $\xi$. \\

\noindent \textbf{Case 2.1.} Let us assume first that $\tilde \varepsilon =1$.\\
 Since $f\neq 0$ and $\lambda\neq 0$, we may define a real function $\theta$ depending on $t$, by $\tan\theta=\frac{f'}{f\lambda}$. Let $\psi$ be the vector field given by 
\begin{equation}\label{psi0}
\psi=\cos\theta\cdot E_1+\sin\theta\cdot \xi
\end{equation} 
and having positive length $\langle \psi, \psi \rangle = 1.$ It follows that 
\begin{align*} 
D_X \psi&=(\cos\theta\cdot \frac{f'}{f}-\sin\theta\cdot\lambda)X\\
		&=0,
\end{align*}
 for any vector field $X$ tangent to $\bar M$, and we compute 
\begin{align*}
D_{E_1}\psi=(\theta'+\mu) (\cos\theta\cdot \xi- \sin\theta \cdot E_1).
\end{align*}
As $\tan \theta= \frac{f'}{f\lambda}$, it follows that $\theta'=\frac{f'' f \lambda-f'^2\lambda-f f'\lambda'}{f^2 \lambda^2+f'^2}$ and, using \eqref{relscurta}, we obtain that $\theta'+\mu=0$ and so,  $D_{E_1}\psi=0$. Therefore, $\psi$ is a constant vector field of length $1$ and  we may assume, without loss of generality, that $\psi \in \mathbb{R}^{n+1}_s$ is given by 
\begin{equation}\label{relpsi}
\psi=(1,0,\ldots,0).\end{equation}
Let $\varphi$ be a unit vector field orthogonal to $\psi$, defined by \begin{equation}\label{varphi0}
\varphi=-\sin\theta \cdot E_1+\cos\theta\cdot \xi.
\end{equation}
One may easily check that 
\begin{align*}
D_{E_1}\varphi &=-(\theta'+\mu) (\cos\theta\cdot  E_1+\sin\theta\cdot \xi)\\
&=0
\end{align*}
and 
\begin{align}\label{dxphi}
D_X\varphi=-(\sin\theta\frac{f'}{f}+\lambda \cos\theta)X.
\end{align}
We have, of course, that $D_X\varphi\neq 0$ and so $\varphi$ is an immersion with $\langle \varphi,\varphi \rangle=1.$ Therefore, we may parameterize it as 
\begin{align}\label{relphi}
\begin{split}
\varphi : \bar M^{n-1}(c) &\longrightarrow \mathbb{S}^{n-1}_{\tilde{s}}(1)\\
(u_1,\ldots, u_{n-1}) &\mapsto (0,y_1,\ldots,y_n),
\end{split}
\end{align}
where $y_1^2+\ldots-y_{n}^2=1$, $\tilde s\leq s$.\\
From the expressions in \eqref{psi0} and \eqref{varphi0}, we recover $E_1$ and by using relations \eqref{relpsi} and \eqref{relphi}  we obtain that
\begin{equation*}
E_1=\frac{\partial F}{\partial t}=(\cos\theta,-\sin\theta y_1,\ldots,-\sin\theta y_n).
\end{equation*}
It follows that there exists a function $G$ such that 
\begin{equation*}
F=(\int\cos\theta dt, -y_1\int\sin\theta dt,\ldots,  -y_n\int\sin\theta dt) +G(u_1,\ldots,u_{n-1})
\end{equation*}
and 
\begin{equation*}
F_{u_i}=\left(0,-\left(\int \sin\theta dt\right) y_{u_i}\right)+G_{u_i},
\end{equation*}
for $y:=(y_1,\ldots,y_{n})$ and $y_{u_i}:=\frac{\partial y}{\partial u_i}$.
We may rewrite \eqref{dxphi} as $D_X{\varphi}=-\frac{\lambda}{\cos\theta}X\equiv -\frac{\lambda}{\cos\theta} dF(X)=-\frac{\lambda}{\cos\theta} D_XF$. Therefore,  
\begin{align*}
D_{\frac{\partial}{\partial u_i}}\varphi &= -\frac{\lambda}{\cos\theta} D_{\frac{\partial}{\partial u_i}}F \quad \Longleftrightarrow\\
(0,y_{u_i})	&= -\frac{\lambda}{\cos\theta}  \left(-\int\sin\theta dt\right) \left(0,y_{u_i}\right)+G_{u_i}.
\end{align*}
It is easy to check that a primitive  $\int \sin\theta dt$ is given by $\frac{\cos\theta}{\lambda}$. With this choice of primitive, one may notice that $G_{u_i}=0$. Therefore, $G$ is constant and we have 
\begin{equation}
F(t, u_1,\ldots,u_{n-1})= ( \int\cos\theta(t)dt, -\int\sin\theta(t)dt\cdot y_1,\ldots, -\int\sin\theta(t)dt\cdot y_n )+G,
\end{equation}
or equivalently, for $f_1(t):=  \int\cos\theta(t)dt $, $f_2(t):=  -\int\sin\theta(t)dt $ and after translating the hypersurface back to the origin, we recover one of the parameterizations in \eqref{primulcaz}, Case 1.

\noindent \textbf{Case 2.2.} Let us assume that $\tilde \varepsilon =-1$.\\
\textbf{Case 2.2.a)} We will first look into the case when  $\left| \frac{f'}{f\lambda}  \right|=1$. When changing $\lambda$ into $-\lambda$ we may also change $\xi$ with $-\xi$ such that we always have $\frac{f'}{f\lambda}=1$, i.e. $\lambda=\frac{f'}{f}$.\\
We remark first that equations \eqref{relscurta} imply that $c=0$ and $f$ satisfies $\mu=\frac{f''}{f'}$.\\
Let  $\alpha=\alpha(t)$ such that $\ln \alpha(t)$ is a primitive of $\mu(t)$ and let  $\psi=\alpha(E_1+\xi)$. We obtain that 
\begin{align*}
D_X\psi=X(\alpha)(E_1+\xi)+\alpha(\frac{f'}{f}-\lambda)X=0
\end{align*} and
\begin{align*}
D_{E_1}\psi &=\alpha_t(E_1+\xi)+\alpha(-\mu \xi-\mu E_1 )\\
 			&=(\alpha_t-\mu\alpha)(E_1+\xi)\\
 			&=0.
\end{align*}
Hence, $\psi$ is a constant vector, of length $0$. Next, we define $\varphi=\frac{1}{\alpha}(-E_1+\xi)$ and compute
\begin{align*}
D_{E_1}\varphi& =-\frac{\alpha_t}{\alpha^2}(-E_1+\xi)+\frac{1}{\alpha}(\mu\xi-\mu E_1)\\
	&=(\frac{\mu}{\alpha}-\frac{\alpha_t}{\alpha^2}) (\xi-E_1)\\
	&=0
\end{align*}
and 
\begin{align*}
D_X\varphi&=X\left(\frac{1}{\lambda}\right)(-E_1+\xi)+\frac{1}{\alpha}(-D_XE_1+D_X\xi)\\
&=-\frac{2\lambda}{\alpha}X.
\end{align*}
Therefore, $\varphi$ is an immersion of dimension $n-1$ given by 
\begin{equation}
\begin{split}
\varphi:M^{n-1}\longrightarrow \mathbb{E}^{n+1}_s,\quad
\varphi=\varphi(u_1,\ldots,u_{n-1})=(y_1,\ldots,y_{n+1}),
\end{split}
\end{equation} with $y_i$ functions of $u_1,\ldots,u_{n-1}$, for $i=1,\ldots, n+1$.
Since $\psi$ is a null constant vector, we can fix it as $\psi=(1,0,\ldots,1)$. Moreover, we have that $\langle \varphi,\psi \rangle =-2$, which gives $y_1=y_{n+1}-2$, and, since $\langle \varphi, \varphi \rangle=0$, it implies that  $y_1=-1+\frac{1}{4}\sum\limits_{i=2}^n(\pm y_i^2)$ and  $y_{n+1}=1+\frac{1}{4}\sum\limits_{i=2}^n(\pm y_i^2)$. Therefore
\begin{equation}\label{phiB1}
\varphi(u_1,\ldots,u_{n})=(-1+\frac{1}{4}\sum\limits_{i=2}^n(\varepsilon_i y_i^2), y_2,\ldots, y_n,  1+\frac{1}{4}\sum\limits_{i=2}^n(\varepsilon_i y_i^2) ).
\end{equation}
Furthermore, we have
\begin{align*}
E_1 &=\frac{1}{2}\left( \frac{1}{\alpha}\psi-\alpha\varphi \right)\\
	&= \frac{1}{2}\left( \frac{1}{\alpha}+\alpha (1-\frac{1}{4}\sum\limits_{i=2}^n(\varepsilon_i y_i^2)), -\alpha y_2,\ldots,-\alpha y_n, \frac{1}{\alpha}-\alpha(1+\frac{1}{4} \sum\limits_{i=2}^n(\varepsilon_i y_i^2) )  \right).
\end{align*}
By integration with respect to $t$, we recover the position vector as
\begin{align}\label{FB1}
\begin{split}
F(t,u_2,\ldots,u_n)=& \frac{1}{2}\left(\int \frac{1}{\alpha} dt + [1-\frac{1}{4}\sum\limits_{i=2}^n(\varepsilon_i y_i^2)] \int \alpha dt, - y_2\int \alpha dt,\ldots,-y_n\int \alpha dt,\right.\\
	&\left. \int \frac{1}{\alpha}dt -[1+\frac{1}{4} \sum\limits_{i=2}^n(\varepsilon_i y_i^2) ] \int\alpha dt ) \right)+G(u_2,\ldots,u_n),
\end{split}
\end{align}for some function $G$.
We have that $ D_{\frac{\partial}{\partial u_i}}\varphi=-\frac{2\lambda}{\alpha} dF\left(\frac{\partial}{\partial u_i} \right)$. Writing this explicitly by using \eqref{phiB1} and \eqref{FB1}, we obtain  that 
\begin{equation}
1=\frac{\lambda}{\alpha}\int\alpha dt+G_{u_i}.
\end{equation}
One may easily check that $\frac{\alpha}{\lambda}$ is a primitive of the form $\int\alpha dt$. With this choice of primitive, it follows that $G$ is a constant.
Therefore, the immersion writes as 
\begin{align}
(t,u_2,\ldots,u_n)\mapsto \int\frac{1}{\alpha} dt  \left(\begin{matrix}
1\\0 \\ \vdots\\ 0\\1\\
\end{matrix} \right) +\int \alpha dt \left(\begin{matrix}
1-\frac{1}{4}\sum\limits_{i=2}^{n}\varepsilon_i y_i^2 \\ -y_2\\ \vdots\\ -y_{n}\\-1-\frac{1}{4}\sum\limits_{i=2}^{n}\varepsilon_i y_i^2
\end{matrix} \right),
\end{align}
which, up to a local diffeomorphism,  corresponds to the parameterization described in section \ref{prelim}, Case 3.

\noindent \textbf{Case 2.2.b)} Assume $\left| \frac{f'}{f\lambda} \right|<1$. Then we may define $\theta$ a function of $t$ such that 
\begin{equation}\label{tanh}
\tanh\theta=\frac{f'}{f \lambda }.
\end{equation}
Let $\psi$ be a vector field defined by
\begin{equation}
\psi=\cosh\theta\cdot E_1+\sinh\theta\cdot\xi
\end{equation}
We compute 
\begin{align*}
D_X\psi=&(\cosh\theta\cdot  \frac{f'}{f}-\sinh\theta\cdot \lambda) X,\\
=& 0
\end{align*}
and \begin{align*}
D_{E_1}\psi=& (\theta'-\mu)(\sinh\theta\cdot  E_1+\cosh\theta\cdot\xi)\\
	=&0,
\end{align*}
since $\theta'-\mu=0$. This follows easily from \eqref{tanh} and \eqref{generaleq}.
Hence, $\psi$ is a constant vector of length $1$ and we can fix it to be $\psi=(1,0,\ldots,0)$.\\
Let $\varphi=-\sinh\theta\cdot E_1+\cosh\theta\cdot \xi$ be a vector field of length $-1$, orthogonal to $\psi$. We compute the covariant derivatives of $\varphi$:
\begin{equation*}
D_{E_1}\varphi=(\theta'-\mu)(\cosh\theta\cdot E_1+\sinh\theta\cdot\xi)=0
\end{equation*}
and
\begin{equation}\label{vardxphi}
D_X\varphi=(\sinh\theta\cdot\frac{f'}{f}-\cosh\theta\cdot \lambda)X.
\end{equation}
As $D_X\varphi\neq 0,\ \forall X$ and $\langle \varphi, \varphi \rangle=-1$, we have that $\varphi$ is an immersion defined as 
\begin{equation}
\varphi: M_1^n(c)\longrightarrow \mathbb{H}^{n-1}_{\tilde s}(-1),\quad 
\varphi(u_1,\ldots,u_n)=(0,y_1,\ldots, y_{n}), 
\end{equation} 
where $\sum\limits_{i=1}^n\varepsilon_i y_i^2=-1$, for $\varepsilon_i=\pm 1$ determined according to the signature of the metric and $\tilde s< s$.\\
We may express $E_1$ as
\begin{equation*}
E_1=\cosh\theta\cdot \psi-\sinh\theta\cdot \varphi=(\cosh\theta, -\sinh\theta\cdot y ), 
\end{equation*} 
for $y=(y_1,\ldots,y_n)$. As $F_t=E_1$, we obtain by integration
\begin{equation*}
F=(\int\cos\theta dt, -\int\sinh\theta dt\cdot y)+G(u_1,\ldots,u_n)
\end{equation*}
and then, it follows that
\begin{equation*}
F_{u_i}=(0, -\int\sin\theta dt\cdot y_{u_i})+G_{u_i}.
\end{equation*}
Moreover, from \eqref{vardxphi} we may write
\begin{align*}
D_{\frac{\partial}{\partial u_i}}\varphi&=-\frac{\lambda}{\cosh\theta}\frac{\partial}{\partial u_i}\Longleftrightarrow \\ 
(0,y_{u_i})&=\frac{\lambda}{\cosh\theta} \int\sinh\theta dt \cdot (0, y_{u_i})+G_{u_i}
\end{align*}
For the particular choice of primitive $ \frac{\cosh\theta}{\lambda}$ for  $ \int\sinh\theta dt$, we see that $G$ is constant and, eventually, the immersion $F$ is given by
\begin{equation}
F=(\int\cos\theta dt, -\int\sinh\theta dt\cdot y), 
\end{equation} where $y=(y_1,\ldots,y_n)$, $\sum\limits_{i=1}^n\varepsilon_i y_i^2=-1$ and $\varepsilon_i=\pm 1$, which corresponds to the rotational hypersurfaces described in section \ref{prelim},  Case 1.

\noindent\textbf{Case 2.2.c)} Assume $\left| \frac{f'}{f\lambda} \right| >1$. Let $\theta$ be defined by 
\begin{equation*}
\coth \theta=\frac{f'}{f \lambda } 
\end{equation*}
and let $\psi$ denote the vector field
\begin{equation*}
\psi=\sinh\theta\cdot E_1+\cosh\theta\cdot \xi.
\end{equation*}
We compute the following derivatives.
\begin{align*}
D_X\psi &=(\sinh \theta \cdot \frac{f'}{f}-\cosh\theta \cdot\lambda)X\\
		&=0,
\end{align*}
and, since $\theta'-\mu=0$, 
\begin{align*}
D_{E_1}\psi &=(\theta'-\mu)(\cosh\theta\cdot  E_1+\sinh\theta\cdot \xi)\\
			&=0.
\end{align*}
Therefore, $\psi$ is a constant vector of length $1$, which we will fix to be $\psi=(1,0,\ldots,0)$.\\
Let $\varphi$ be a unit length vector field orthogonal to $\psi$, defined as
\begin{equation}
\varphi= \cosh\theta \cdot E_1+\sinh\theta\cdot\xi.
\end{equation}
We see that $\varphi$ is independent of $t$, since 
\begin{equation}
D_{E_1}\varphi=(\theta'-\mu)(\sinh\theta \cdot E_1+\cosh\theta \xi)=0.
\end{equation}
Moreover, 
\begin{equation}\label{eticheta}
D_X\varphi=(\cosh\theta\cdot \frac{f'}{f}-\sinh\theta \cdot \lambda)X
\end{equation}
and, of course, $D_X\varphi\neq 0,\forall X$. Hence, $\varphi$ is an immersion with position vector of length $1$ defined as
\begin{equation}
\varphi:\bar M^{n-1}(c)\to \mathbb{S}^{n-1}_{\tilde s}(1),\quad \varphi(u_1,\ldots,u_n)=(0,y_1,\ldots,y_n),
\end{equation}
where $\sum\limits_{i=1}^n \varepsilon_i y_i^2=1$, for $\varepsilon_i=\pm1$ according to the signature of the metric and $\tilde s\leq s$. \\
We have that 
\begin{align*}
E_1 =&-\sinh\theta\cdot\psi+\cosh\theta\cdot\varphi\\
	=& (-\sinh\theta,\cosh\theta\cdot y), \text{ for } y=(y_1,\ldots,y_n).
\end{align*}
Therefore, by integration with respect to $t$, we recover 
$$F(t,u_1,\ldots,u_n)=(-\int\sinh\theta dt, (\int\cosh\theta dt)\cdot y)+G(u_1,\dots,u_n).$$
From \eqref{eticheta} we may write $D_X\varphi=\frac{\lambda}{\sinh\theta}X $,
which implies 
\begin{align*}
D_{\frac{\partial}{\partial u_i}}\varphi&=\frac{\lambda}{\sinh\theta} D_{\frac{\partial}{\partial u_i}}F \quad \Longleftrightarrow\\
(0,y_{u_i})&=\frac{\lambda}{\sinh\theta}\left(0,(\int\cosh\theta dt )y_{u_i} \right)+G_{u_i}.
\end{align*}
We choose the primitive $\frac{\sinh\theta}{\lambda}$ for $\int\cosh\theta dt$ and we see that $G$ must be a constant. Therefore, the parameterization of the immersion $F$ becomes 
$$F(t,u_1,\ldots,u_n)=(-\int\sinh\theta dt, (\int\cosh\theta dt)\cdot y), \text{ where }y=(y_1,\ldots,y_n),$$ 
which corresponds to the rotation hypersurface described in section \ref{prelim}, Case 1.


\noindent {\bf Acknowledgements.}
M. Moruz is supported by a grant of the Romanian Ministry of
Education and Research, CNCS-UEFISCDI, project number PN-III-P1-1.1-PD-2019-0253, within PNCDI III.

\end{document}